\documentclass{amsart}
\usepackage{amsfonts,amsmath,amssymb,amscd,latexsym,graphicx,epsfig}
\usepackage[all]{xy}

\def\west{{\rm(\raisebox{-0.5mm}{\includegraphics[width=3mm]{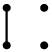}}\rm)}}
\def\south{{\rm(\raisebox{-0.5mm}{\includegraphics[width=3mm]{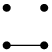}}\rm)}}
\def\east{{\rm(\raisebox{-0.5mm}{\includegraphics[width=3mm]{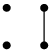}}\rm)}}
\def\north{{\rm(\raisebox{-0.5mm}{\includegraphics[width=3mm]{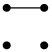}}\rm)}}
\def\sw{{\rm(\raisebox{-0.5mm}{\includegraphics[width=3mm]{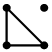}}\rm)}}
\def\nw{{\rm(\raisebox{-0.5mm}{\includegraphics[width=3mm]{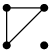}}\rm)}}
\def\se{{\rm(\raisebox{-0.5mm}{\includegraphics[width=3mm]{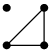}}\rm)}}
\def\ne{{\rm(\raisebox{-0.5mm}{\includegraphics[width=3mm]{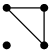}}\rm)}}
\def\ns{{\rm(\raisebox{-0.5mm}{\includegraphics[width=3mm]{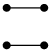}}\rm)}}
\def\sn{{\rm(\raisebox{-0.5mm}{\includegraphics[width=3mm]{nslines.eps}}\rm)}} 
\def\ew{{\rm(\raisebox{-0.5mm}{\includegraphics[width=3mm]{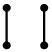}}\rm)}}
\def\mdiag{{\rm(\raisebox{-0.5mm}{\includegraphics[width=3mm]{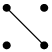}}\rm)}}
\def\adiag{{\rm(\raisebox{-0.5mm}{\includegraphics[width=3mm]{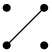}}\rm)}}
\def\square{{\rm(\raisebox{-0.5mm}{\includegraphics[width=3mm]{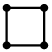}}\rm)}}
\begin{document}
\title[The conjugacy problem in the 4-strand braid group]{A fast solution to the conjugacy problem in the 4-strand braid group}
\author{Matthieu Calvez and Bert Wiest}
\address{Matthieu Calvez / Bert Wiest, IRMAR (UMR 6625 du CNRS), Universit\'e de Rennes 1,
Campus de Beaulieu, 35042 Rennes Cedex, France}
\email{matthieu.calvez@univ-rennes1.fr, bertold.wiest@univ-rennes1.fr}


\begin{abstract}
We present an algorithm for solving the conjugacy search problem in the four strand braid group. The computational complexity is cubic with respect to the braid length.
\end{abstract}
\maketitle
\newtheorem{prop}{Proposition}[section]
\newtheorem{propdef}[prop]{Proposition/Definition}
\newtheorem{lem}[prop]{Lemma}
\newtheorem{calcul}[prop]{Calculation}
\newtheorem{thm}[prop]{Theorem}
\newtheorem{defi}[prop]{Definition}
\newtheorem{ex}[prop]{Example}
\newtheorem{qu}[prop]{Question}
\newtheorem{coro}[prop]{Corollary}
\newtheorem{remark}[prop]{Remark}

\section{Introduction}
The \emph{conjugacy problem} is one of the three famous decision problems in groups first formulated by Dehn in the early 20th century. The aim is to decide whether two given elements~$x$ and~$y$ of a group~$G$ are conjugate in~$G$, i.e.\ if there exists an element~$z$ in~$G$ such that  $x=z^{-1}yz$ (which we shall denote $x=y^z$). If so, then an additional problem is to actually search for such a conjugating element~$z$. These two problems are called CDP (conjugacy decision problem) and CSP (conjugacy search problem).


We know since Garside \cite{garside} that CDP and CSP are solvable for the braid groups~$B_n$, meaning that there exists an algorithm for solving these two problems in~$B_n$, $n\geqslant 1$. In fact, the properties of braid groups discovered in~\cite{garside} are now known to hold for a large class of groups, called Garside groups \cite{dehornoy}; this class of groups contains for instance all Artin-Tits groups of spherical type~\cite{BrSa}.


A \emph{Garside group} of finite type~$G$ is equipped with a partial order relation~$\preccurlyeq$, called the \emph{prefix order}, which is invariant under left multiplication and which induces on~$G$ a lattice structure (i.e.\ every pair of elements has a largest common divisor and a least common multiple). Moreover, the positive cone of this relation  $P=\{x\in G\ |\  1\preccurlyeq x\}$ contains a special element~$\Delta$, called the \emph{Garside element}, with the following properties: firstly, conjugation by~$\Delta$ stabilizes~$P$, and secondly, the set of positive divisors of~$\Delta$ with respect to~$\preccurlyeq$ (also called the set of \emph{simple elements}) is finite and generates~$G$. An important property of Garside groups is the existence of a \emph{left normal form}~\cite{dehornoy}. This means that for any element~$x$ of the group there exists a unique decomposition of the form $x=\Delta^px_1\ldots x_r$, where $p\in\mathbb Z$, 
$r\in\mathbb N$, 
and the factors~$x_i$ are simple elements. In this factorization,
the quantity~$\ell(x)=r$ is called the \emph{canonical length} of~$x$.

The braid groups admit in fact two distinct Garside structures (i.e.\ two distinct pairs $(\preccurlyeq,\Delta)$). We shall use both structures. On the one hand, the \emph{classical} Garside structure, stemming from Garside's original article \cite{garside}: for $x,y\in B_n$, $x\preccurlyeq y$ if and only if $x^{-1}y$ can be written as a product of Artin generators $\sigma_i$, all with positive powers; the Garside element is $\Delta=(\sigma_1\ldots\sigma_{n-1})\ldots(\sigma_1\sigma_2)\sigma_1$. On the other hand, the \emph{dual} (or BKL) structure, introduced by Birman, Ko and Lee in~\cite{bkldual}, which we shall describe in detail in Section~\ref{S:Prolegomenes}.


Since Garside, several more and more powerful algorithms for solving CDP and CSP have been proposed \cite{em,gebhardt,gebhardtgm}. We briefly recall that each of the algorithms \cite{em,gebhardt,gebhardtgm} for solving CDP and CSP in a Garside group~$G$ is based on the calculation, for any given $x\in G$, of a finite non-empty subset~$E_x$ of the conjugacy class of~$x$, satisfying $E_x=E_y$ if and only if~$x$ and~$y$ are conjugate. The \emph{Super Summit Sets} ($SSS$) \cite{em}, the \emph{Ultra Summit Sets} ($USS$) \cite{gebhardt} and the \emph{sets of sliding circuits} ($SC$) \cite{gebhardtgm} are three examples of such characteristic subsets. Unfortunately, despite the very high speed (in practice) of the most recent algorithms, the existence of a polynomial bound on the algorithmic complexity (with respect to the length of the input) is still an open problem, even in the case of the braid groups.


The main result of the current article is the following theorem:


\begin{thm}\label{T:main}             															
There exists an algorithm which solves CDP and CSP in the braid group~$B_4$ and whose algorithmic complexity is cubic with respect to the length of the input braid words.
%
\end{thm}                                                                                                                                                                 											    %

We are not able to prove Theorem~\ref{T:main} using only the tools of Garside theory. We shall also use a geometric point of view on braids.
It is known (see e.g.~\cite{birman}) that the braid group~$B_n$ ($n\geqslant 1$) can be identified with the mapping class group of the $n$-times punctured disk~$\mathbb D_n$. In this context, braids can be classified according to their dynamical properties, in the following trichotomy (Nielsen-Thurston classification) \cite{casson,fathi}: a braid~$x$ is

%
%
%
\begin{itemize}
\item \emph{periodic}, if there exists an integer~$m$ such that $x^m\in ZB_n=\left<\Delta^2\right>$,
\item \emph{reducible}, if there exists a non-empty family~$\mathcal F$ (called the canonical reduction system) of isotopy classes of nondegenerate disjoint simple closed curves in~$\mathbb D_n$ (non-degenerate means not null-homotopic, not homotopic into a puncture and not boundary-parallel) such that
\begin{itemize}
\item the $x$-action leaves $\mathcal F$ invariant,
\item $\mathcal F$ does not intersect any other isotopy class of simple closed curve in~$\mathbb D_n$ which is invariant under some power of~$x$
\end{itemize}
\item \emph{pseudo-Anosov} (pA) otherwise.
\end{itemize}

%

We remark that the definition of ``reducible'' most frequently found in the literature also encompasses certain periodic elements. In this paper we only apply the word ``reducible'' to the braids which would usually be called ``reducible non-periodic''.

The paper~\cite{bggm1} proposes a program, based both on the Nielsen-Thurston classification and on Garside theory, for solving CDP and CSP in the braid groups \emph{in polynomial time} with respect to both the length of the input braid words and their number of strands. A first step in this program is the construction of a polynomial time algorithm for deciding the dynamical type of any given braid (Open question~1 in~\cite{bggm1}).


In \cite{calvezwiest}, the authors answered this question in the case of the group~$B_4$: they produced an algorithm of complexity~$O(\ell^2)$ to decide the Nielsen-Thurston type of any given 4-strand braid of length~$\ell$. Thus in the group~$B_4$, in order to solve CDP and CSP it is sufficient to solve these problems for pairs of elements which are known to be of the same dynamical type (as pairs of braids of different dynamical type are never conjugate).


The algorithm given in \cite{calvezwiest} also implies a solution to CDP and CSP for \emph{reducible} four-strand braids of length at most~$\ell$ in time $O(\ell^2)$. The main lemma here is that for braids with at most 3~strands and of length at most~$\ell$, the problems CDP and CSP are solvable in time $O(\ell^2)$, see \cite{calvezwiest}.


The case of \emph{periodic} braids is treated in~\cite{bggm3}, where an algorithm of complexity $O(\ell^3n^2\log n)$ for solving CDP and CSP for periodic braids with~$n$ strands of canonical length at most~$\ell$ is presented.


In order to prove Theorem~\ref{T:main}, we thus have to produce an algorithm of complexity at most~$O(\ell^3)$ capable of solving CDP and CSP for \emph{pseudo-Anosov} four-strand braids of length at most~$\ell$. Technically, our main contribution is the following result, which gives a partial affirmative answer (in the special case of pseudo-Anosov 4-strand braids) to the Open Question 2 in~\cite{bggm1}: using the vocabulary of~\cite{gebhardtgm}, if a 4-strand pseudo-Anosov braid is \emph{rigid} (meaning, roughly speaking, that the normal form is as simple as possible), then its set of sliding circuits (and also its Ultra Summit Set) is ``small'':


\begin{thm}\label{cardinalbound}
For every braid $x$ in $B_4$ which is pseudo-Anosov and rigid with respect to the dual Garside structure, the cardinality of $SC(x)$ for the dual structure is bounded above by~$O(\ell(x)^2)$.
%
\end{thm}																							%

This result implies that the algorithm given in \cite{gebhardtgmalgo} for solving CDP and CSP has complexity $O(\ell^3)$ when applied to two 4-strand braids which are of length at most~$\ell$, pseudo-Anosov and rigid in the dual structure.


The rest of the proof of Theorem~\ref{T:main} thus has to consist of a reduction to the rigid case.
A key result in this direction is Theorem 3.37 in \cite{bggm1}, which states that for any braid~$x$ (with arbitrarily many strands) there exists a strictly positive integer~$m$ such that $x^m$ is conjugate to a rigid braid. Moreover, for any fixed number of strands, this integer~$m$ is bounded by a constant which does not depend on~$x$ (see Theorem~\ref{bggm1}).


The next step in the reduction to the rigid case is provided by the result from~\cite{gmroots} that for any $m\in\mathbb N$, any pseudo-Anosov braid has \emph{at most one} $m$th root. Therefore, for any two pseudo-Anosov braids $x,y\in B_n$, for any positive integer~$m$, and for any braid $z\in B_n$, the relations $x=y^z$ and $x^m=(y^m)^z$ are equivalent. Thus the only remaining problem is the following: for any pair $x,y$ of pseudo-Anosov braids we have to produce (in polynomial time) a suitable power~$s$ and rigid conjugates $\bar x$ of~$x^s$ and~$\bar y$ of~$y^s$ (together with conjugating elements).


The existence of a polynomial time algorithm which, for any given pseudo-Anosov braid, constructs a rigid conjugate (if one exists), together with a conjugating element, is proved in~\cite{calvez} (see Proposition~\ref{bla} below). This algorithm is based on the linear conjugator bound for pseudo-Anosov elements in mapping class groups due to Masur and Minsky \cite{masur}. We deduce the following algorithm:


\begin{thm}\label{algosecond}
There is an algorithm with the following propoerties:
\begin{itemize}
\item as input, it takes two pseudo-Anosov braids $x,y\in B_n$ of canonical length at most~$\ell$,
\item as output, it yields an integer~$s$ and $n$-strand braids $\bar x$, $\bar y$, $z_1$ and $z_2$ such that $\bar x$ and $\bar y$ are rigid and satisfy $\bar x=(x^s)^{z_1}$ and $\bar y=(y^s)^{z_2}$,
\item for any fixed~$n$, the complexity is $O(\ell^2)$.
\end{itemize}																

\end{thm}																							%

We can now describe the algorithm promised by Theorem~\ref{T:main}:


{\bf{ALGORITHM:}} \\
INPUT: $x$ and $y$ two elements of the four-strand braid group.\\
OUTPUT: whether or not~$x$ and~$y$ are conjugate, and if they are, an element~$z\in B_4$ so that $x=y^z$.
\begin{itemize}
\item[(1)] Determine the dynamical types of $x$ and $y$, using~\cite{calvezwiest}. If they are not the same, answer ``$x$ and~$y$ are not conjugate'' and STOP.
\item[(2)] If $x$ and $y$ are periodic use \cite{bggm3} and STOP.
\item[(3)] If $x$ and $y$ are reducible, use \cite{calvezwiest} and STOP.
\item[(4)] If $x$ and $y$ are pseudo-Anosov, use the algorithm of Theorem~\ref{algosecond} in order to produce $s,\bar x,\bar y,z_1,z_2$ with the required properties.
\item[(5)] Apply Algorithm 3 of \cite{gebhardtgmalgo} to $\bar x$ and $\bar y$.  If~$\bar x$ and~$\bar y$ are conjugate, then this algorithm produces $c\in B_4$ such that $\bar x={\bar y}^c$. In this case answer ``$x$ is conjugate to~$y$ by $z_2cz_1^{-1}$." and STOP.
\item[(6)]  Answer ``$x$ and~$y$ are not conjugate".
\end{itemize}


This paper is organized as follows. In Section~\ref{S:Prolegomenes} we recall some prerequisites from Garside theory. In Section~\ref{ProofTheorem2} we prove Theorem~\ref{cardinalbound} bounding the size of the sets of sliding circuits. In Section~\ref{ProofTheorems13} we prove Theorem~\ref{algosecond} and finally Theorem~\ref{T:main}.

{\bf{Acknowledgements.}} The authors are grateful to Juan Gonz\'{a}lez-Meneses for suggesting a simplification in the proof of Theorem \ref{cardinalbound}. 
The first-named author was partially supported by a grant from R\'{e}gion Bretagne, by MTM2010-19355 and FEDER.


\section{Prerequisites from Garside theory}\label{S:Prolegomenes}

In this section we first recall some general facts which apply to all Garside groups. Then we turn our attention to rigid elements and describe in detail the structure of the sets of sliding circuits in this particular case. Finally, we recall briefly the dual Garside structure on the 4-strand braid group. Note that none of the results in this section are new; however, we shall introduce in the second and third part some non-standard notation which will be useful for the rest of the paper.



\subsection{Reminders on Garside theory}


Throughout this subsection, $G$ denotes a Garside group. Its partial order relation is denoted~$\preccurlyeq$, and the associated positive cone is~$P$. The order~$\preccurlyeq$ is a lattice: the greatest common divisor of two elements~$x,y$ of~$G$ is denoted $x\wedge y$. The set of divisors of the Garside element~$\Delta$ which lie in~$P$ is finite and generates~$G$; its elements are called the \emph{simple} elements. We denote~$\tau$ the interior automorphism associated to~$\Delta$; it preserves~$P$ and the relation~$\preccurlyeq$. In particular, the automorphism~$\tau$ induces a permutation of the (finite) set of simple elements, and since these elements generate~$G$, the interior automorphism~$\tau$ is of finite order.


We also recall that every simple element~$s$ possesses a \emph{right complement} $\partial(s)$ defined by the formula $\partial(s)=s^{-1}\Delta$. This notion allows us to define:


\begin{defi}
Let $s_1$ and $s_2$ be two simple elements of~$G$. We say the pair $s_1 s_2$ is \emph{left-weighted} if $\partial(s_1)\wedge s_2=1$, in other words if $s_1$ is the greatest simple divisor of $s_1 s_2$.
%
\end{defi}

\begin{prop} \cite{dehornoy,em}
Let $x\in G$. There exists a unique decomposition $x=\Delta^px_1\ldots x_r$, where~$r$ is a non-negative integer, $p$ is the greatest integer satisfying $\Delta^p\preccurlyeq x$, the $x_i$ are simple elements with $x_r\neq 1$, and (if $r\geqslant 2$) the pair $x_i x_{i+1}$ is left-weighted for $i=1,\ldots,r-1$.
%
\end{prop}

In the previous proposition, the decomposition $x=\Delta^px_1\ldots x_r$ is called the (left) normal form of~$x$; the integers~$p$ and~$r$ are called the \emph{infimum} and the \emph{canonical length} of~$x$, and they are denoted $\inf(x)$ and $\ell(x)$. The \emph{supremum} $\sup(x)$ is the quantity $p+r$. For every element~$x$ of $G$, $\sup(x)=\min\{k\in \mathbb Z, x\preccurlyeq \Delta^k\}$. We observe that the elements of canonical length zero are precisely the powers of~$\Delta$; these elements are as simple as possible within their conjugacy class.


\begin{remark}\label{remarklength}\rm
The set of simple elements, taken as a generating set of~$G$, induces a length function on~$G$: the length~$|x|$ of an element~$x$ of~$G$ is by definition the smallest possible length of a word representing~$x$ whose letters are simple elements or their inverses. We note that always $\ell(x)\leqslant |x|$. We also have the following relations, for any~$x$ satisfying~$\inf(x)=p$ and $\ell(x)=r$ \cite{dehornoy,charneymeier}:
$$|x|=\begin{cases}
p+r & \text{if $p\geqslant 0$},\\
r     & \text{if $p<0$ and $|p|\leqslant r$,}\\
|p|  & \text{if $p<0$ and $|p|>r$.}
\end{cases}$$
%
\end{remark}

\begin{defi}
Let $x\in G$, and let $x=\Delta^p x_1\ldots x_r$ be the normal form of~$x$. Suppose~$r\geqslant 1$. We call
\begin{itemize}
\item \emph{initial factor} of $x$ the simple element $\iota(x)=\tau^{-p}(x_1)$,
\item \emph{final factor} of $x$ the simple element $\varphi(x)=x_r$.
\end{itemize}
\end{defi}

We recall that~$G$ is equipped with different operations which are defined in terms of normal forms, each corresponding to a particular conjugation.

\begin{defi}\cite{em}
Let $x\in G$ with normal form $x=\Delta^px_1\ldots x_r$. Suppose $r\geqslant 1$. We define:
\begin{itemize}
\item the \emph{cycling} $\mathbf c(x)=x^{\iota(x)}=\Delta^p x_2\ldots x_r \tau^{-p}(x_1)$,
\item the \emph{decycling} $\mathbf d(x)=x^{\varphi(x)^{-1}}=\Delta^p \tau^p(x_r)x_1\ldots x_{r-1}$.
\end{itemize}
If $\ell(x)=0$, we also define $\mathbf c(x)=\mathbf d(x)=x$.
\end{defi}

Note that cycling and decycling commute with the automorphism~$\tau$. These two operations make it possible to ``simplify'' elements of~$G$ within their conjugacy class:


\begin{prop}\cite{em}\label{cyclagedecyclage}
Let~$x\in G$.
\begin{itemize}
\item[(i)] The subset of the conjugacy class of~$x$ consisiting of all elements with minimal canonical length is finite and non-empty. Its elements have \emph{simultaneously} maximal infimum and minimal supremum. This subset is called the \emph{Super Summit Set} of $x$, and denoted $SSS(x)$.
\item[(ii)]  There exist $k_0,l_0\in \mathbb N$ such that for every $k\geqslant k_0$ and $l\geqslant l_0$, $\mathbf c^k(\mathbf d^l(x))\in SSS(x)$.
\end{itemize}
\end{prop}


More recently, Gebhardt and Gonz\'{a}lez-Meneses introduced a new type of conjugation which combines cycling and decycling into a single, conceptually simpler, operation:


\begin{defi}\label{cyclicsliding}\cite{gebhardtgm}
Let $x\in G$ with normal form $x=\Delta^px_1\ldots x_r$. Suppose $r\geqslant 1$. We define the \emph{preferred prefix} of~$x$ by the formula $\mathfrak p(x)=\iota(x)\wedge \partial(\varphi(x))$. \emph{Cyclic sliding} is the operation~$\mathfrak s$ defined by
$$\mathfrak s(x)=x^{\mathfrak p(x)}.$$
If $\ell(x)=0$ then we also define $\mathfrak s(x)=x$.
\end{defi}


Now the analogue of Proposition~\ref{cyclagedecyclage}~(ii), also proved in \cite{gebhardtgm}, states that for every $x\in G$, there exists an integer~$k_0$ such that $\mathfrak s^k(x)\in SSS(x)$ whenever $k\geqslant k_0$. Moreover, the observation that $\mathfrak s$ preserves $SSS(x)$ implies that the set of periodic points of~$\mathfrak s$ in the conjugacy class of~$x$ is a (finite) nonempty subset of~$SSS(x)$; this is another conjugacy invariant:


\begin{defi}
Let $x\in G$. The \emph{set of sliding circuits} of~$x$ is the set of all conjugates of~$x$ which are periodic points of the cyclic sliding operation. That is,
$SC(x)=\{y\in x^G\ |\  \exists k\in \mathbb N, \ \mathfrak s^k(y)=y\}$.
%
\end{defi}

An important example of fixed points of~$\mathfrak s$ are the so-called rigid elements:

\begin{defi}
Let $x\in G$ with normal form $x=\Delta^px_1\ldots x_r$. Suppose $r\geqslant 1$. We say~$x$ is  \emph{rigid} if the pair $\varphi(x)\iota(x)$ is left-weighted.
\end{defi}

In particular, an element of canonical length 0 is not rigid.

One very useful quality of the Super Summit Set is that it can quickly be reached by iterated cyclic sliding:

\begin{thm}\cite{bklsssbound,gebhardtgm}\label{bklbound}
Let~$G$ be a Garside group. Then there exists a constant~$\alpha$, which depends only on the group ~$G$ and its Garside structure, such that for every ${x\in G}$, $\mathfrak s^{\ell(x)\cdot \alpha}(x)\in SSS(x)$.
%
\end{thm}

For instance in the case of the classical Garside structure on~$B_n$ we have $\alpha=\frac{n(n-1)}{2}$, and $\alpha=n-1$ for the dual structure. This result has a very important algorithmic application: it yields an algorithm of complexity $O(\ell^2)$ for calculating an element $y\in SSS(x)$ for any given $x\in B_n$ in normal form of canonical length $\ell(x)=\ell$. The algorithm can even output an explicit conjugating element between~$x$ and~$y$.

For the rest of the paper, we shall mostly be dealing with braids which lie in their own Super Summit Set (because pushing braids into their own SSS only costs $O(\ell^2)$, as we have just seen).


In order to prove Theorem~\ref{cardinalbound}, we need to study the structure of the set of rigid conjugates of any element~$x\in G$. This is the object of the next subsection.



\subsection{Ultra Summit Sets and rigid elements}\label{elementsrigides}

Our aim in this subsection is to describe the structure of the set of all rigid conjugates of an element of a Garside group. We use the same notation as in the previous subsection. Our study is based mainly on the following proposition from~\cite{gebhardtgm}:


\begin{prop}\cite{gebhardtgm}
Let $x\in G$. Suppose that~$x$ has a rigid conjugate. Then $SC(x)$ is precisely the set of all rigid conjugates of~$x$.
\end{prop}


The following results are also proven in~\cite{gebhardtgm} (independently of the existence of a rigid conjugate):


\begin{defi}\cite{gebhardtgm}
Let~$x\in G$ and $y\in SC(x)$. A simple, non-trivial element~$s$ of~$G$ is said to be a \emph{minimal arrow} for~$y$ if $y^s\in SC(x)$ and if the only positive prefixes~$t$ of~$s$ with $y^t\in SC(x)$ are $t=1$ and $t=s$.
%
\end{defi}

\begin{prop}\cite{gebhardtgm}\label{cyclage}
Let~$x\in G$.
\begin{itemize}
\item[(i)] The set $SC(x)$ is stable under conjugation by $\Delta$, by cycling, and by decycling.
\item[(ii)] (See also \cite{bggm2}). For every $y\in SC(x)$, the minimal arrows for~$y$ are prefixes of~$\iota(y)$ or of~$\partial(\varphi(y))$.
\end{itemize}
%
\end{prop}


\begin{defi}\cite{gebhardtgm}
To every element~$x$ of~$G$ we associate a connected, oriented graph $SCG(x)$ describing the set $SC(x)$ as follows:
\begin{itemize}
\item the graph has one vertex for every element of $SC(x)$,
\item for every element $y$ of $SC(x)$ and every minimal arrow~$s$ for~$y$, the graph $SCG(x)$ has an oriented edge from the vertex~$y$ to the vertex~$y^s$. This edge is labelled~$s$.
%
\end{itemize}
\end{defi}

When~$x$ has a rigid conjugate, the graph $SCG(x)$ has a particularly elegant structure, which we describe now. For the rest of this subsection we shall always suppose that~$x$ is rigid.

\begin{defi}
Let $x\in G$ be a rigid element, and let $y\in SC(x)$. The \emph{orbit} of~$y$ is the set
$O_y=\{\tau^{k}\mathbf c^l(y)\; |\; k,l\in \mathbb N\}$.
%
\end{defi}

\begin{lem}\label{orbit}
Let $x\in G$ be a rigid element, and let $y\in SC(x)$. Denote $p=\inf(x)$ and $r=\ell(x)$.
\begin{itemize}
\item[(i)] The orbit~$O_y$ is a subset of $SC(x)$.

\item[(ii)] The orbit~$O_y$ is stable under cycling, decycling, and $\tau$; in particular, for every~$z\in O_y$, $z^{\iota(z)}$ and $z^{\partial(\varphi(z))}$
are element of~$O_y$.

\item[(iii)] Let $y_1,y_2\in SC(x)$. Then $O_{y_1}\neq O_{y_2}$
if and only if $O_{y_1}\cap O_{y_2}=\emptyset$.

\item[(iv)] The cardinality of the orbit~$O_y$ is bounded above by~$f \cdot\ell(y)$, where~$f$ is the order of~$\tau$).
\end{itemize}
\end{lem}

The results of Lemma~\ref{orbit} are well-known to the experts. For the sake of completeness, we nevertheless give a proof.

\begin{proof}
(i) This is a consecquence of Proposition~\ref{cyclage}~(i).

(ii) The stability of~$O_y$ under cycling and conjugation by~$\Delta$ follows immediately from the definition (since~$\mathbf c$ and~$\tau$ commute). Let us now prove that~$\mathbf d(z)\in O_y$ whenever $z\in O_y$. Let $\Delta^p z_1\ldots z_r$ be the normal form of~$z$. Then $\mathbf d(z)=\Delta^p\tau^p(z_r)z_1\ldots z_{r-1}$, and this expression is in normal form since~$z$ was rigid. Moreover, we observe that $\mathbf d(z)=\tau^p(\mathbf c^{r-1}(z))$ (still due to the rigidity), and this last element belongs to~$O_y$ by definition. For the second part of the statement, we only need to recall that $z^{\iota(z)}=\mathbf c(z)$ and that $z^{\partial(\varphi(z))}=\tau(\mathbf d(z))$.


{\bf{Remark.}} Let~$f$ be the order of the automorphism~$\tau$. For any $m\in \mathbb N$ and any $z\in SC(x)$ we now observe that $\mathbf c^{rm}(z)=\tau^{-pm}(z)$, and in particular
$$\mathbf c^{rfm}(z)=z.$$

(iii) It suffices to prove that for every $y\in SC(x)$ and every $z\in O_y$, $O_z=O_y$.
The inclusion $O_z\subset O_y$ holds by part~(ii). Conversely, let $k,l\in \mathbb N$ such that $z=\tau^k\mathbf c^l(y)$; we shall prove that~$y\in O_z$. Let~$j$ be the smallest positive integer satisfying~$rfj>l$. Then due to the above remark we have
$$\mathbf c^{rfj-l}(\tau^{k(f-1)}(z))=y$$ so~$y\in O_z$.
Finally by~(ii), $O_y\subset O_z$, as desired.

(iv) We consider the~$r$ first successive cyclings of the~$f$ elements $y, \tau(y),\ldots, \tau^{f-1}(y)$. This yields at most~$rf$ elements of~$O_y$. We claim that it yields \emph{all} elements of~$O_y$. Indeed, let $z=\tau^k(\mathbf c^l(y))$. Let $m\in \mathbb N$ be such that $rm\leqslant l<r(m+1)$. Then $l-rm\in [0,\ldots, r-1]$, and
$$z=\tau^k(\mathbf c ^{rm}(\mathbf c^{l-rm}(y)))=\tau ^{k-pm}(\mathbf c^{l-rm}(y))$$
where the second equality follows from the above remark.
Thus~$z$ occurs as one of the first~$r$ cyclings of one of $y,\tau(y),\ldots, \tau^{f-1}(y)$.
\end{proof}

We have proven that the relation~$\sim$, defined by $x\sim y$ if and only if~$O_x=O_y$, is an equivalence relation on~$SC(x)$, and~$SC(x)$ is the disjoint union of the different orbits~$O_y$. We denote $\widetilde{SC}(x)$ the quotient set $SC(x)/\hspace{-2mm}\sim$. We now associate a ``quotient graph'' $\widetilde{SCG}(x)$ to $\widetilde{SC}(x)$ in the same way as~$SCG(x)$ is associated to~$SC(x)$. In order to do this rigorously, we need the following definition:

\begin{defi}
Let $y\in SC(x)$ be rigid, and let~$s$ be a minimal arrow for~$y$. We say~$s$ is a minimal \emph{useful} arrow if $y^s\notin O_y$.
\end{defi}

\begin{remark}\label{strict}\rm
According to Proposition~\ref{cyclage}~(ii) and Lemma~\ref{orbit}~(ii), the minimal useful arrows for~$y$ are \emph{strict} prefixes of~$\iota(y)$ or of~$\partial(\varphi(y))$.
\end{remark}

We recall the notion, due to Gebhardt~\cite{gebhardt}, of the \emph{transport under cycling} of an arrow: if $y, s\in G$, we define the transport under cycling of~$s$ at~$y$ by the formula $s_y^{(1)}=\iota(y)^{-1}s\iota(y^s)$. It is known (\cite{gebhardt}, Corollary 2.7) that the transport induces a bijection between the set of minimal arrows for $y\in SC(x)$ and the set of minimal arrows for~$\mathbf c(y)$. Similarly, conjugation by~$\Delta$ induces a bijection between the minimal arrows for~$y$ and the minimal arrows for~$\tau(y)$. In particular, if~$s$ is a minimal useful arrow between~$y$ and~$y^s$, then $s^{(1)}_{y}$ is a minimal useful arrow between~$\mathbf c(y)$ and~$\mathbf c(y^s)$, and~$\tau(s)$ is a minimal useful arrow between~$\tau(y)$ and $\tau(y^s)$. Thus we can define the desired quotient graph without any ambiguity (i.e.\ the arbitrary choices made in the following definition do not matter):


\begin{defi}
To every rigid element~$x$ of~$G$ we associate a connected, oriented graph~$\widetilde{SCG}(x)$ as follows:
\begin{itemize}
\item The vertices of the graph correspond to elements of $\widetilde{SC}(x)$,
\item For every element $O_y\in \widetilde{SC}(x)$, we arbitrarily choose a representative~$y'$ of~$O_y$. Now to any minimal useful arrow~$s$ for~$y'$, from~$y'\in O_y$ to~$z'\in O_z$, we associate an edge of the graph, oriented from~$O_y$ to~$O_z$.
\end{itemize}
\end{defi}

In order to bound the size of~$SC(x)$, it suffices to bound the number of vertices of~$\widetilde{SCG}(x)$: if $\widetilde{SCG}(x)$ has~$k$ vertices, then the cardinality of~$SC(x)$ is at most $k\cdot f \cdot \ell(x)$ (due to Lemma~\ref{orbit}~(iv)).


\subsection{The dual structure of $B_4$}\label{dual4}

A detailed account of the dual Garside structure on braid groups can be found in the original article~\cite{bkldual}, and an introduction in Chapter VIII of~\cite{ddrw}.
%
We restrict ourselves here to a brief description of this structure in the case of the four-strand braid group~$B_4$. We consider the sub-monoid $BKL_4^{+}$ of~$B_4$ generated by the braids~$a_{p,q}$, $1\leqslant p<q\leqslant 4$, where
%
%
$$a_{p,p+1}=\sigma_p\ \text{ for }\  p=1,\ldots, 3,$$
%
$$a_{1,3}=\sigma_2^{-1}\sigma_1\sigma_2,$$
$$a_{2,4}=\sigma_3^{-1}\sigma_2\sigma_3,$$
$$a_{1,4}=\sigma_3^{-1}\sigma_2^{-1}\sigma_1\sigma_2\sigma_3.$$
The notation $BKL$ is derived from the names of the discoverers of this structure: Birman, Ko and Lee. The monoid~$BKL_4^+$ induces a partial order relation on~$B_4$: $x\preccurlyeq y$ if and only if $x^{-1}y\in BKL_4^{+}$. Taking as Garside element the braid $\delta=\sigma_1\sigma_2\sigma_3$, these data give rise to a new Garside structure, which we denote $BKL_4$. For instance, we shall write $x\in BKL_4$ in order to say that~$x$ is a four-strand braid seen in the structure~$BKL_4$, and given as a product of the generators~$a_{i,j}$.

We now introduce some notation concerning the $BKL_4$-structure which we shall be using for the rest of the article. We recall that~$B_4$ can be seen as the mapping class group of the four times punctured disk~$\mathbb D_4$. In the context of the $BKL_4$-structure it is practical to parametrize~$\mathbb D_4$ as the unit disk in~$\mathbb C$ with punctures $P_j=\frac{1}{2}e^{-\frac{i(2j-1)\pi}{4}}$, for $j=1,\ldots, 4$. The braid~$a_{p,q}$ then corresponds to the counterclockwise half Dehn-twist along the arc $(P_p,P_q)$. Pictorially, we will represent the braid~$a_{p,q}$ by the segment $(P_p,P_q)$; for instance, $a_{2,4}$ is denoted~$\adiag$, $a_{1,4}$ is written~$\east$, and so on. Similarly, the braid which cyclically exchanges $P_3$, $P_2$ and $P_1$ by a counterclockwise movement is denoted~$\sw$. With this notation, the generators~$a_{p,q}$ are subject to the following relations:
%
%
$$\east\west=\west\east=\ew, \hspace{2mm} \north\south=\south\north=\ns,$$
$$\south\west=\west\mdiag=\mdiag\south=\sw, \hspace{2mm} \west\north=\north\adiag=\adiag\west=\nw,$$
$$\north\east=\east\mdiag=\mdiag\north=\ne, \hspace{2mm} \east\south=\south\adiag=\adiag\east=\se.$$
The Garside element is~$\delta=\square$. Conjugation by~$\delta$ corresponds to a one-quarter counterclockwise turn, and~$\tau$ is an automorphism of order~4 of~$B_4$. Therefore, Lemma~\ref{orbit}(iv), applied to the $BKL_4$-structure, states that the orbit of a rigid braid~$x$ contains at most $4\cdot\ell(x)$ elements.


Our proof of Theorem~\ref{cardinalbound} is based on the simplicity of the lattice of simple elements of~$BKL_4$. It has only~14 elements:

%
$$1,\south,\west,\north,\east,\mdiag,\adiag,\sw,\nw,\ne,\se,\ew,\ns,\delta.$$

The relations listed above are length-preserving. This allows us to define a morphism $\lambda\colon\thinspace B_4\longrightarrow \mathbb Z$ by sending every braid~$a_{p,q}$ to~1. For any braid~$x$, we call $\lambda(x)$ the \emph{weight} of~$x$. As already remarked in~\cite{bklsssbound}, $\lambda(\delta)=3$ and for any other simple nontrivial element~$s$ we have $\lambda(s)=1$ or~2. This observation yields a new quantity, in addition to canonical length, supremum and infimum, which is constant inside the Super Summit Set: 


\begin{lem}\label{invariantSSS}
Let $x\in BKL_4$, and let $y\in SSS(x)$. For every $z\in SSS(x)$, the normal form of~$z$ contains as many factors of weight~2 and as many factors of weight~1 as the normal form of~$y$.
%
\end{lem}
\begin{proof}
For every braid $x\in BKL_4$, if $k_1$ is the number of factors of weight~1 and~$k_2$ the number of factors of weight~2 in the normal form of~$x$, then $\ell(x)=k_1+k_2$ and $\lambda(x)=3\inf(x)+2k_2+k_1$. Thus~$k_1$ and~$k_2$ are constant in the Super Summit Set, since canonical length, weight, and infimum are constant there.
\end{proof}

The following very simple remark will turn out to be very useful:

\begin{remark}\label{lem}\rm
Let~$a$ and~$b$ be two simple elements for~$BKL_4$. If~$a$ is of weight~2 and~$\delta$ does not divide the product~$ab$, then~$a.b$ is in normal form.
\end{remark}

We finally claim that in the $BKL_4$ structure, the existence of a minimal useful arrow~$s$ from~$y'\in O_y$ to $z'\in O_z$ is equivalent to the existence of a minimal useful arrow from~$z'\in O_z$ to some element of~$O_y$. (In other words, every edge in~$\widetilde{SCG}(x)$, for $G=BKL_4$, is oriented both ways.) Let us prove this claim. According to Remark~\ref{strict}, such a minimal arrow~$s$ is a \emph{strict} prefix either of~$\iota(y')$ or of~$\partial(\varphi(y'))$. In particular,~$\lambda(s)=1$. Now there is an arrow, which is of weight~1 and thus minimal, given in the first case by~$s^{-1}\iota(y')$ from~$z'$ to $\mathbf c(y')$, and in the second case by~$s^{-1}\partial(\varphi(y'))$, from~$z'$ to~$\tau\mathbf d(y')$.



\section{Proof of Theorem \ref{cardinalbound}}\label{ProofTheorem2}

Throughout this section, we use the dual Garside structure on~$B_4$. Our aim is to prove Theorem~\ref{cardinalbound}, so we consider a rigid pseudo-Anosov braid~$x$, and we try to bound the size of $SC(x)$. The hypothesis that~$x$ is pseudo-Anosov implies that the canonical length of~$x$ is strictly larger than~1 (this can be proven by analysing all braids with canonical length~1).

We shall see that it suffices to prove Theorem~\ref{cardinalbound} separately in three special cases, which are defined in terms of the simple factors occurring in~$x$ (see Section~\ref{dual4}). We will consider successively the following three cases:
%
%
\begin{itemize}
\item The normal form of~$x$ contains at least one factor of weight~1 and one factor of weight~2 -- this case is solved in Proposition~\ref{P:FacteursPoids2}. (Notice that all other elements of~$SC(x)$ will have the same property, by Lemma~\ref{invariantSSS}).
\item There exists an element~$y$ of~$SC(x)$ such that all the factors other than~$\delta$ occurring in the normal form of~$y$ belong to~$\{\west,\north,\east,\south\}$ -- this case is solved in~Proposition \ref{T:RigidEdge}.
\item For every element~$y$ of~$SC(x)$, all the factors other than~$\delta$ occurring in the normal form of~$y$ are of weight~1, and at least one of them is~$\mdiag$ or~$\adiag$ -- this case is solved in Proposition~\ref{linesize}.
\end{itemize}
In the first two cases, the hypothesis that~$x$ should be pseudo-Anosov is in fact unnecessary. In these cases, we even construct a \emph{linear} bound on~$\#SC(x)$. The third case requires much more sophisticated techniques, and gives rise to an example showing that the quadratic bound is optimal.


\subsection{A simple special case}

We now describe a simple special case where Theorem~\ref{cardinalbound} can be proved by elementary arguments.

\begin{prop}\label{P:FacteursPoids2}
Let~$x\in BKL_4$ be a rigid braid whose normal form contains at least one factor of weight~1 and at least one factor of weight~2. Then the set~$SC(x)$ consists only of~$O_x$ and in particular $\#SC(x)\leqslant 4\cdot\ell(x)$.
\end{prop}

The proof of this proposition is based on the following lemma:

\begin{lem}
Let~$x\in BKL_4$ be a rigid braid whose normal form contains at least one factor of weight~1 and at least one factor of weight~2. Let $y\in SC(x)$. Then there is no strict prefix of~$\iota(y)$ or of $\partial(\varphi(y))$ which is a minimal arrow for~$y$.
\end{lem}

\begin{proof}
By Lemma~\ref{invariantSSS}, the normal form of every element of~$SSS(x)$ (and thus of~$SC(x)$) contains at least one factor of weight~1 and at least one factor of weight~2. Therefore, if $\delta^py_1\ldots y_r$ is the normal form of~$y$, then there exist $k,l$ with $1\leqslant k,l\leqslant r$ such that~$\lambda(y_k)=1$ et $\lambda(y_l)=2$. Let~$t$ be a simple nontrivial braid such that $t\prec \tau^{-p}(y_1)$ or $t\prec \partial(y_r)$ and $y^t\in SSS(x)$. We shall prove that~$y^t$ cannot be rigid. Up to replacing~$y$ by its inverse, which is rigid with $\iota(y^{-1})=\partial(\phi(y))$, we can suppose that~$t\prec\tau^{-p}(y_1)$.


Necessarily, $\lambda(t)=1$ and $\lambda(y_1)=2$ (because a factor~$y_1$ with $\lambda(y_1)=1$ cannot have any strict prefix), so we can take $l=1$. Let us choose~$k$ as the smallest index of a simple factor of weight~1. Thus for $1\leqslant j<k$ we have $\lambda(y_j)=2$. Let us study the element
$$y^t=t^{-1}\delta^py_1\ldots y_rt.$$
Since~$y$ was rigid, the pair $y_r.t$ is in normal form.
Next, we are going to prove that $\varphi(y^t)=t$. Writing $t_1=\tau^p(t)$, the calculation of the normal form of~$y^t$ can be performed using $k-1$ successive local slidings:
%
$$(t_j^{-1}y_j)y_{j+1}=(t_j^{-1}y_jt_{j+1})(t_{j+1}^{-1}y_{j+1})=y'_j(t_{j+1}^{-1}y_{j+1})\ \text{ for}\ \ j=1,\ldots, k-1,$$
where $t_{j+1}=\partial(t_j^{-1}y_j)\wedge y_{j+1}$.
For all $j=1,\ldots,k-1$, we have $t_j\neq1$ (i.e.\ the pair $(t_j^{-1}y_j)y_{j+1}$ is not in normal form as written) because otherwise we'd obtain $\sup(y^t)>\sup(y)$. Moreover,~$\delta$ is not a prefix of $t^{-1}_1y_1\ldots y_rt$. Therefore, $\lambda(t_j)=1$ for $j=1,\ldots,k-1$.
Since $\lambda(y_k)=1$, we also have $t_{k}=y_k$, and $y'_{k-1}$ is of weight~2. Now by Remark~\ref{lem}, the pair $y'_{k-1}\cdot y_{k+1}$ (with $y_{k+1}=t$ if $k=r$) is in normal form.
This completes the proof of our claim that~$\varphi(y^t)=t$.

We can now prove that~$y^t$ is not rigid. Indeed,
%
$$\iota(y^t)\wedge \partial(\varphi(y^t))=\tau^{-p}(t_1^{-1}y_1t_2)\wedge \partial(t)=t^{-1}\iota(y)\tau^{-p}(t_2)\wedge \partial(t).$$
Thus~$t^{-1}\iota(y)$ is a nontrivial common prefix of~$\iota(y^t)$ and~$\partial(\varphi(y^t))$, so the pair $\varphi(y^t)\iota(y^t)$ is not left-weighted. In summary, the braid $y^t$ is not rigid and~$t$ was not a minimal arrow for~$y$.
\end{proof}

Now, the proof of Proposition~\ref{P:FacteursPoids2} is an immediate consequence of the preceding lemma and of Lemma~\ref{orbit}(iv).

Now, in order to prove Theorem~\ref{cardinalbound}, we have to find a quadratic bound on the size of $SC(x)$ for any rigid pseudo-Anosov braid~$x\in BKL_4$. By Proposition~\ref{P:FacteursPoids2}, we can restrict our attention to braids whose normal form has all its factors (other than~$\delta$) of the same weight (1 or 2). Up to considering inverses, we can restrict ourselves to the case of weight~1 (see \cite{bggm2}, Corollary 3.10).

So for the rest of the proof of Theorem~\ref{cardinalbound}, we can suppose that~$x$ is a rigid pseudo-Anosov braid whose normal form has only factors of weight~1 (i.e.\ $\south$, $\west$, $\north$, $\east$, $\mdiag$, $\adiag$), and $\delta^{\pm 1}$. By Lemma~\ref{invariantSSS}, all elements of~$SSS(x)$ have the same property. Moreover, using Remark~\ref{strict}, we see that for every $y\in SC(x)$, all possible minimal useful arrows for~$y$ are strict prefixes of~$\partial(\varphi(y))$ (there is no strict non-trivial prefix of~$\iota(y)$ because~$\lambda(\iota(y))=1$). In particular, all vertices of~$\widetilde{SCG}(x)$ have valence at most~3.


We make one more simple, but very useful general observation:

\begin{lem}\label{diag}
Suppose that the normal form of the rigid braid $y\in SC(x)$ has only factors of weight~1 (and~$\delta^{\pm 1}$), with at least one factor equal to~$\adiag$ or to~$\mdiag$. Then the vertex~$O_y$ of~$\widetilde{SCG}(x)$ is at most bivalent.
%
\end{lem}

\begin{proof}
Up to replacing~$y$ by~$y'\in O_y$, we can suppose that~$\varphi(y)=\mdiag$.
But $\partial(\mdiag)=\ns$, and this simple element has only two strict positive prefixes.
\end{proof}

We split the rest of our argument into two parts. In Subsection~\ref{calculatoire}, we study the case where~$SC(x)$ contains an element that does not satisfy the hypotheses of Lemma~\ref{diag}, i.e. an element whose normal form contains, apart from~$\delta^{\pm 1}$, only the letters from $\{\south,\west,\north,\east\}$; we shall denote the latter set~$\mathcal E$. By contrast, Subsection~\ref{USSbivalent} deals with the case where all elements of~$SC(x)$ satisfy the hypotheses of Lemma~\ref{diag}.


\subsection{Some element of~$SC(x)$ has all its factors in~$\mathcal E$}\label{calculatoire}

We recall the notation $\mathcal E=\{\west,\north,\east,\south\}$. Our aim in this subsection is to prove the following result, whose proof is elementary but involves a lot of careful case-checking:


\begin{prop}\label{T:RigidEdge}
Let~$x\in BKL_4$ be a rigid braid. Let us suppose that~$SC(x)$ contains some element~$y$ whose normal form has all of its factors (apart from~$\delta^{\pm 1}$) belonging to~$\mathcal E$. Then the graph~$\widetilde{SCG}(x)$ has at most six vertices. Moreover, $\#SC(x)\leqslant 24\cdot\ell(x)$.
\end{prop}


%
%
%

The last sentence of Proposition~\ref{T:RigidEdge} follows immediately from the preceding one, together with Lemma~\ref{orbit}~(iv).

First we note that in order to prove Proposition~\ref{T:RigidEdge}, it suffices to prove that for some non-zero integer $m\in\mathbb N$ the graph~$\widetilde{SCG}(x^m)$ has at most 6 vertices. Indeed, since all braids in~$SC(x)$ are rigid, there is an injection from~$\widetilde{SC}(x)$ to~$\widetilde{SC}(x^m)$, sending an orbit~$O_y$ to an orbit~$O_{y^m}$.


So possibly after replacing~$x$ by~$x^4$, we can suppose that $\inf(x)$ is a multiple of~4. In fact, since for any integer~$m$, multiplication by~$\delta^{4m}$ induces an isomorphism between~$SC(x)$ and~$SC(\delta^{4m}x)$, we can even suppose that $\inf(x)=0$ (and thus that the infimum of any element of~$SSS(x)$ is zero).


So for the rest of the proof of Proposition~\ref{T:RigidEdge}, we shall assume that for~$y$ (and hence for all elements of~$O_y$) the normal form has all letters belonging to~$\mathcal E$.


\begin{remark}\label{RemarkEdge}\rm
Conjugation by~$\delta$ induces a permutation of~$\mathcal E$. Moreover, for all $s,t\in\mathcal E$, the product $st$ is in normal form if and only if $t\in\{s,\tau(s)\}$.
\end{remark}

Remark~\ref{RemarkEdge} allows us to give a precise description of the normal form of~$y$:

\begin{lem}\label{FormNormal}
Let $y\in BKL_4$ be a rigid braid with $\inf(y)=0$, all of whose factors belong to~$\mathcal E$. Then, possibly after replacing~$y$ by another element of~$O_y$, the normal form of~$y$ is of the form
%
%
$$y=\prod_{j=1}^r \tau^{-r+j}\left(\west^{k_j}\right),$$
where the $k_j$, $j=1,\ldots, r$ are strictly positive integers, and~$r=1$ or $r\equiv 0\pmod 4$.
\end{lem}

\begin{proof}
Up to conjugating~$y$ by a power of~$\delta$, we can suppose that~${\varphi(y)=\west}$. By Remark~\ref{RemarkEdge} and our hypothesis on~$y$, the normal form of~$y$ is indeed a product of the form
$$y=\tau^{-(r-1)}\left(\west^{k_{1}}\right)\ldots \west^{k_r}$$
for integers~$r$ and $k_1,\ldots,k_{r}$ all strictly positive. Then, due to rigidity and Remark~\ref{RemarkEdge}, we have $\iota(y)=\varphi(y)$ or $\iota(y)=\tau\left(\varphi(y)\right)$. Let us suppose that $r>1$. Up to cycling, we can suppose $\iota(y)=\tau\left(\varphi(y)\right)$, which means that~$\tau^{-r+1}\left(\west\right)=\tau\left(\west\right)$.
This implies that $r\equiv 0\pmod 4.$
\end{proof}

\begin{lem}
If $r=1$ in Lemma~\ref{FormNormal}, then $\#SC(x)=6$.
\end{lem}
\begin{proof}
If~$r=1$ then
$SC(x)=\{\west^{k_1},\north^{k_1},\east^{k_1},\south^{k_1},\mdiag^{k_1},\adiag^{k_1}\}.$
\end{proof}

\begin{lem}\label{SommetTrivalent}
Suppose that $r>1$ in Lemma~\ref{FormNormal}. Then there exists a minimal arrow for~$y$ if and only if~$r\equiv 0\pmod 3$. If this is the case, then~$y$ admits in fact three minimal (but not necessarily useful) arrows. If not, then the graph $\widetilde{SCG}(x)$ has a single vertex.
\end{lem}
\begin{proof}
According to Lemma~\ref{FormNormal}, we have $r\equiv 0\pmod 4$, and we can rewrite
$$y=\prod_{j=1}^{m}\left(\south^{k_{j,1}}\east^{k_{j,2}}
\north^{k_{j,3}}\west^{k_{j,4}}\right):=\prod_{j=1}^m\alpha_j,$$
with $m:=\frac{r}{4}$ and $k_{j,i}>0$ for all~$j,i$ with $1\leqslant j\leqslant m$ and $1\leqslant i\leqslant 4$. The minimal useful arrows for~$y$, if they exist, are all strict prefixes of
$\partial(\west)$, so they are  $\north, \east$ or $\mdiag$.

The proof of the lemma essentially comes down to the following calculations, where the right hand sides of the equations (except for their first factor) are always in normal form; in other words,  $A_j$, $B_j$ and $C_j$ are normal forms, independently of the powers occurring in the formulae (this calculation uses the notation~$\alpha_j$ defined in the previous paragraph):

$$\alpha_j\north=
\mdiag\left(\south\mdiag^{k_{j,1}-1}\east^{k_{j,2}}\north\west^{k_{j,3}}\adiag^{k_{j,4}-1}\right):=\mdiag A_j,$$
$$\alpha_j\mdiag=
\east\left(\south\east^{k_{j,1}}\adiag^{k_{j,2}-1}\north^{k_{j,3}}\west\south^{k_{j,4}-1}\right):=\east B_j,$$
$$\alpha_j\east=
\north\left(\south^{k_{j,1}}\east\north^{k_{j,2}}\mdiag^{k_{j,3}-1}\west^{k_{j,4}}\right):=\north C_j.$$

We remark that, independently of~$j$ and of the powers occurring, the ``pairs'' $A\cdot C$, $B\cdot A$ and $C\cdot B$ are in normal form. On the one hand, if $r\equiv 1$ or $r\equiv 2\pmod 3$, then this shows that for every~$u$ with $u\prec \partial(\west)$ we have $u\nprec yu$, and in particular $y^u\notin SSS(x)$. On the other hand, if $m\equiv 0\pmod 3$, then this shows that the three braids
%
$$y^{\north}=\left(\prod_{j=1}^{\frac{m}{3}}\alpha_{3j-2}\alpha_{3j-1}\alpha_{3j}\right)^{\north}=
\prod_{j=1}^{\frac{m}{3}}C_{3j-2}B_{3j-1}A_{3j},$$
$$y^{\east}=\prod_{j=1}^{\frac{m}{3}}B_{3j-2}A_{3j-1}C_{3j}\ \ \ \text{ and }\ \ \
y^{\mdiag}=\prod_{j=1}^{\frac{m}{3}}A_{3j-2}C_{3j-1}B_{3j}$$
are rigid.
\end{proof}

We suppose from now on that $r\equiv 0\pmod 3$ (this is always satisfied up to replacing~$x$ by~$x^3$). Let~$u$ be such that $u\prec\partial(\west)$. Then it follows from our proof of Lemma~\ref{SommetTrivalent} that we can always find, up to cyclic permutation of the factors, an element~$z$ of~$O_{y^u}$ of the form
$$z=\prod_{j=1}^{\frac{m}{3}}C_{3j-2}B_{3j-1}A_{3j},$$
by making an appropriate choice of indices and powers inside the factors.

We can then rewrite~$y$ in the form
$$y=\prod_{\nu=1}^{\frac{m}{3}}\south^{a_{\nu}}\east^{b_{\nu}}\north^{c_{\nu}}\west^{d_{\nu}}
\south^{e_{\nu}}\east^{f_{\nu}}\north^{g_{\nu}}\west^{h_{\nu}}
\south^{i_{\nu}}\east^{j_{\nu}}\north^{k_{\nu}}\west^{l_{\nu}},$$

with strictly positive integers $a_{\nu}$, $b_{\nu},\ldots, l_{\nu}$ for all $\nu=1,\ldots, \frac{m}{3}$, and then~$z$ becomes
%
$$z=\prod_{\nu=1}^{\frac{m}{3}}\left[
\south^{a_{\nu}}\east\north^{b_{\nu}}\mdiag^{c_{\nu}-1}\west^{d_{\nu}}
\south\east^{e_{\nu}}\adiag^{f_{\nu}-1}\north^{g_{\nu}}\west\south^{h_{\nu}}\mdiag^{i_{\nu}-1}\right.$$
$$\left.\east^{j_{\nu}}\north\west^{k_{\nu}}\adiag^{l_{\nu}-1}\right].$$

\begin{lem}\label{P:PointExtremal}
If the normal form of~$z$ contains a factor~$\mdiag$ or~$\adiag$, then~$z$ admits a unique minimal useful arrow, i.e. the vertex~$O_z$ of the graph~$\widetilde{SCG}(x)$ is extremal.
\end{lem}

\begin{proof}
Up to cycling or conjugating by~$\delta$ we can suppose that the last factor of~$z$ is~$\adiag$, and that~$l_{\frac{m}{3}}>1$. There are at most two minimal arrows for~$z$, namely~$\east$ and~$\west$. A calculation of the normal form of~$z\west$ shows that~$\west\nprec z{\west}$, which implies that~$z^{\west}\notin SSS(x)$, and hence the lemma.

In order to perform this calculation, we make three observations. Firstly,
$$\adiag^{l_{\nu}-1}\west=\north\adiag\north^{l_{\nu}-2}.$$
Secondly, for arbitrary integers $a,b>0$,
$$\left(\east^a\north\west^b\right)\north=\west\left(\east^a\north\west\adiag^{b-1}\right)$$
and the first factor on the right hand side is independant of the powers~$a$ and~$b$. Thirdly, the pair~$\mdiag\west$ is in normal form. Now the previous calculation can be pushed towards the left along the normal form of~$z$, getting twisted by a conjugation by~$\Delta$ at each step, until it hits, possibly, a factor~$\mdiag$ or $\adiag$, where it gets stuck.

This shows that the multiplication of~$z$ by~$\west$ on the right can only modify the beginning of the normal form of~$z$ if~$\adiag^{l_{\frac{m}{3}}-1}$ is the only occurrence of~$\adiag$ or~$\mdiag$ in~$z$. Moreover, if this is the case, then the initial factor of~$z\west$ is~$\sn$. This completes the proof.
\end{proof}

\begin{lem}\label{PointExtremal}
Suppose that the normal form of~$z$ does not contain any factor~$\adiag$ or~$\mdiag$. Then
\begin{itemize}
\item[(i)] the three strict prefixes of~$\partial(\varphi(z))$ are minimal arrows for~$z$,
\item[(ii)] if~$v$ is a minimal useful arrow for~$z$ conjugating~$z$ to another rigid braid whose normal form contains no factor~$\adiag$ or~$\mdiag$, then $z^v\in O_y$ and $v=\east$.
\end{itemize}
\end{lem}

\begin{proof}
According to our hypothesis, we can further rewrite the formulae from the proof of Lemma~\ref{SommetTrivalent}:
$$y=\prod_{\nu=1}^{\frac{m}{3}}\south^{a_{\nu}}\east^{b_{\nu}}\north\west^{d_{\nu}}
\south^{e_{\nu}}\east\north^{g_{\nu}}\west^{h_{\nu}}
\south\east^{j_{\nu}}\north^{k_{\nu}}\west$$
and 
$$z=\prod_{\nu=1}^{\frac{m}{3}}
\south^{a_{\nu}}\east\north^{b_{\nu}}\west^{d_{\nu}}
\south\east^{e_{\nu}}\north^{g_{\nu}}\west
\south^{h_{\nu}}\east^{j_{\nu}}\north\west^{k_{\nu}}.$$
(i) According to Lemma~\ref{SommetTrivalent}, $z$ admits three minimal (not necessarily useful) arrows.

(ii) Let~$v$ be a minimal useful arrow for~$z$ such that~$z^v$ contains no factor~$\mdiag$ or~$\adiag$. We first mention that at least one such an arrow exists, because~$O_y\neq O_z$.
We know that~$v\in\{\north,\east,\mdiag\}$. Thus it is sufficient to prove that $v\neq \north$ and $v\neq \mdiag$. We are going to apply the formulae from the proof of Lemma~~\ref{SommetTrivalent}, now with~$z$ playing the r\^ole previously played by~$y$.

If~$v=\north$, then the formulae from the proof of Lemma~\ref{SommetTrivalent}, together with the restriction that~$z^v$ must not contain any factors~$\mdiag$ ni $\adiag$, imply the equalities $b_{\nu}=e_{\nu}=h_{\nu}=k_{\nu}=1$. But then $z=z^{\north}$, contradicting the usefulness of~$v$. Thus~$v\neq \north$.


Analogously, if~$v=\mdiag$, then due to the formulae from the proof of Lemma~\ref{SommetTrivalent} we obtain $a_{\nu}=d_{\nu}=g_{\nu}=j_{\nu}=1$. But then
$$z=\prod_{\nu=1}^{\frac{m}{3}}
\south\east\north^{b_{\nu}}\west
\south\east^{e_{\nu}}\north\west
\south^{h_{\nu}}\east\north\west^{k_{\nu}}$$
and
$$z^{\mdiag}=\prod_{\nu=1}^{\frac{m}{3}}
\south\east\north\west^{b_{\nu}}
\south\east\north^{e_{\nu}}\west
\south\east^{h_{\nu}}\north\west\south^{k_{nu}-1}.$$
We obtain $\mathbf c(z^{\mdiag})=\tau(z)$, contradicting the usefulness of~$v$. Thus~$v\neq \mdiag$.
\end{proof}

Lemma~\ref{PointExtremal} shows that the graph $\widetilde{SCG}(x)$ cannot contain a chain of 3 vertices whose elements contain no factor~$\mdiag$ or~$\adiag$. By Lemma~\ref{P:PointExtremal} any vertex which does contains at least one factor~$\mdiag$ or~$\adiag$, but which is adjacent to a vertex which doesn't, is monovalent. Since the graph $\widetilde{SCG}(x)$ is connected, this implies that it has at most 6 vertices. This completes the proof of Proposition~\ref{T:RigidEdge}.



\subsection{All elements of $SC(x)$ have at least one factor not belonging to~$\mathcal E$}
\label{USSbivalent}

In this subsection we suppose that all elements of $SC(x)$ have at least one factor in their normal form equal to~$\mdiag$ ou $\adiag$. According to Lemma~\ref{diag}, the graph $\widetilde{SCG}(x)$ is then a (possibly closed) line. In order to prove Theorem~\ref{cardinalbound}, we need to bound the length of this line. This task seems much more difficult than in the previous subsections, we have currently no elementary proof of Theorem~\ref{cardinalbound} under the above hypotheses. In order to illustrate the difficulty, we show first that the quadratic bound of Theorem~\ref{cardinalbound} is optimal. The following example was obtained with the help of the program GAP \cite{gap}:


\begin{ex}\rm
For all $k\in \mathbb N$, the braid
$\beta_k=\north\west\south\mdiag\east\se^{3k}\south^{-3k}$, whose normal form is
$$\beta_k=\north\west\south\mdiag\east\left[\adiag\south\east\right]^k$$
is rigid and pseudo-Anosov with~$\ell(\beta_k)=3k+5$. Moreover, the graph $\widetilde{SCG}(\beta_k)$ is a line with $3k+2$ vertices. (Explicitely, in order to obtain braids representing all vertices of $\widetilde{SCG}(\beta_k)$, it suffices to conjugate $\beta_k$ by $\south^j$, for $j=0,\ldots,3k+1$.) Thus $\#SC(\beta_k)=4\cdot(3k+2)\cdot(3k+5)$.
\end{ex}
%
%

The proof of Theorem~\ref{cardinalbound} under the hypotheses of this subsection is based on a deep result due to Masur and Minsky (\cite{masur}, Theorem 7.2), namely the linear bound on the length of an element conjugating two pseudo-Anosov elements of a mapping class group.

We consider the length function~$|.|$ on~$B_4$ induced by taking as generators of~$B_4$ the set of divisors of~$\delta$, i.e.\ of BKL-simple braids (see Remark~\ref{remarklength}). The result of Masur and Minsky, applied to the case of 4-strand braids, then states:

\begin{thm}[\cite{calvez}, Proposition 7] \label{masurbraids}
There exists a constant~$c$ such that for every pair $(z_1,z_2)$ of conjugate pseudo-Anosov 4-strand braids, there exists a conjugating element~$w$ (i.e.\ $z_1^w=z_2$) such that $|w|\leqslant c\cdot(|z_1|+|z_2|)$.
\end{thm}

We remark that the length function used in the statement of (\cite{calvez}, Proposition~7) is the length associated to the alphabet of divisors of~$\Delta$, i.e.\ the set of simple braids in the \emph{classical} Garside structure. However, the length functions associated to different finite generating sets in a group are in bilipschitz correspondence. More explicitely, our two length functions on~$B_4$ are related, with the obvious notations, by the formula:
$$|x|_{BKL_4}\leqslant 2\cdot |x|_{\text{classical}}\leqslant 6\cdot |x|_{BKL_4}.$$

In order to complete the proof of Theorem~\ref{cardinalbound}, it is now sufficient to prove the following result (where the constant~$c$ is the one promised by Theorem~\ref{masurbraids}).

\begin{prop}\label{linesize}
Let $x\in BKL_4$ be a rigid pseudo-Anosov braid. Suppose that all elements of~$SC(x)$ have at least one factor of their normal form equal to~$\mdiag$ or~$\adiag$. Then the graph~$\widetilde {SCG}(x)$ has at most $16\cdot c\cdot\ell(x)$ vertices. Thus, $\#SC(x)\leqslant 64\cdot c\cdot\ell(x)^2$.
\end{prop}

\begin{proof}
First we can suppose that $|x|\leqslant 2\cdot\ell(x)$. In order to see this, we notice that multiplying~$x$ by any power~$m$ of the central element~$\delta^4$ induces an isomorphism between the graphs $\widetilde{SCG}(x)$ and $\widetilde{SCG}({\delta^{4m}}x)$. In this way, we can suppose that $\inf(x)\in \{-3,-2,-1,0\}$. Then from Remark~\ref{remarklength} we obtain $|x|\leqslant 2\cdot \ell(x)$ (recalling that $\ell(x)\geqslant 2$, since~$x$ is pseudo-Anosov).

According to Lemma~\ref{diag}, every vertex of the graph $\widetilde{SCG}(x)$ is at most bivalent, so topologically the graph is either a compact line segment or a circle. We claim that any two distinct vertices~$O_a, O_b$ in the graph $\widetilde{SCG}(x)$ can be connected in the graph by a path of length at most $8\cdot c\cdot\ell(x)$. Before proving this claim, we observe that the claim, together with Lemma~\ref{orbit} (iv), implies Proposition~\ref{linesize} (the factor~2 comes from the possibility that the graph might form a circle).

So let~$O_a$ and~$O_b$ be two distinct vertices of $\widetilde{SCG}(x)$, and let~$z_a$ and~$z_b$ be representatives of these two orbits. Due to Theorem~\ref{masurbraids}, there exists a braid~$w$ satisfying $z_a^w=b_b$, and such that $|w|\leqslant 2\cdot c\cdot |x|\leqslant 4\cdot c\cdot\ell(x)$. Up to changing the representative~$z_a$ we can suppose that~$\inf(w)=0$. Then $\lambda(w)\leqslant 2\cdot|w|$, as every factor of the normal form of~$w$ contributes at most~$2$ to the weight of~$w$. Thus~$w$ is the product of at most $2\cdot|w|$ minimal arrows, which yields  a path of length at most~$2\cdot|w|$ between~$O_a$ and~$O_b$ in the graph~$\widetilde{SCG}(x)$.
\end{proof}

\begin{qu} \rm
Open question~2 in~\cite{bggm1} concerns the
existence of a polynomial bound in~$n$ and~$\ell$ on the size of the set of sliding circuits of a rigid (pseudo-Anosov) braid with~$n$ strands and of canonical length at most~$\ell$. Prasolov gave a negative answer, by exhibiting a family of rigid pseudo-Anosov braids for which the size of the SC grows exponentially as a function of~$n$ (for both structures, dual and classical). On the other hand, if we fix~$n$ then no such counter-example is known, and indeed in the special case $n=4$ our Theorem~\ref{cardinalbound} gives an affirmative answer. So we formulate the following question: for any fixed integer~$n$, does there exist a polynomial~$P_n$ such that the cardinality of the (classical or dual) SC of a rigid pseudo-Anosov braid with~$n$ strands is bounded above by $P_n(\ell(x))$?
\end{qu}

\begin{qu}\rm
Is the size of the (classical or dual) SSS of a rigid pseudo-Anosov 4-braid~$x$ bounded above by $P(\ell(x))$, for some polynomial~$P$? We know from~\cite{caruso} that for braids with five or more strands, the size of the classical SSS can increase exponentially with the length of the braid.
\end{qu}


\section{Proofs of Theorems~\ref{algosecond} and~\ref{T:main}}\label{ProofTheorems13}

In this section we will prove Theorems~\ref{algosecond} and~\ref{T:main}. The plan is to prove Theorem~\ref{algosecond} first, and then to prove the validity of the algorithm described in the Introduction and to analyse its complexity.

First we recall one of the main results of~\cite{bggm1}:

\begin{thm}\cite{bggm1}.\label{bggm1}
Let $x\in B_n$ be a pseudo-Anosov braid. Then there exists an integer~$m$ such that $x^m$ is conjugate to a rigid braid. Moreover, the integer~$m$ can be bounded independently of the length of~$x$: $m<{(\frac{n(n-1)}{2}})^3$  for the classical Garside structure and $m<(n-1)^3$ for the dual structure.
\end{thm}

We are going to use a second time the Masur-Minsky linear conjugacy bound, by invoking the following result from~\cite{calvez} whose proof relies on this bound. (Recall that~$\mathfrak s$ denotes the cyclic sliding operation -- see Definition~\ref{cyclicsliding}).


\begin{prop}\label{bla}(\cite{calvez}, Theorem 2).
There exists a constant~$C$, depending only on~$n$ and on the chosen Garside structure, with the following property: if~$x\in B_n$ is a pseudo-Anosov braid lying in its own Super Summit Set, and if~$x$ possesses some rigid conjugate, then the conjugate~$\mathfrak s^{C|x|}(x)$ is rigid.
\end{prop}

This proposition yields a quadratic time algorithm for finding a rigid conjugate~$y$ of any given pseudo-Anosov braid~$x$ satisfying $x\in SSS(x)$, and also for finding a conjugating element, provided a rigid conjugate exists at all.

We are now ready to prove Theorem~\ref{algosecond}.

\begin{proof}[Proof of Theorem~\ref{algosecond}. ]
Let us denote~$\beta(n)$ the upper bound on~$m$ in the statement of Theorem~\ref{bggm1}.
Let~$x,y\in B_n$ be pseudo-Anosov braids. Our aim is to algorithmically find rigid conjugates of~$x^s$ and~$y^s$ for some~$s\in\mathbb N$.

Due to Theorem~\ref{bggm1}, there exist two integers~$i_x$ and~$i_y$, both smaller than~$\beta(n)$, such that~$x^{i_x}$ and~$y^{i_y}$ are conjugate to rigid braids. For all $i=1,\ldots, \beta(n)-1$ simultaneously, our algorithm iterates the operation~$\mathfrak s$ starting from~$x^i$, until a rigid braid is found. The corresponding power~$i_x$ and a braid~$z_x$ such that~$(x^{i_x})^{z_x}$ is rigid are memorized. We denote~$\tilde{x}$ this rigid conjugate of~$x^{i_x}$
The same procedure, applied to~$y$, yields an integer~$i_y$ and braids $z_y$ and~$\tilde y$ with the corresponding properties.

Note that the algorithm so far is doable in time $O(\ell^2)$, where $\ell$ is the maximum of the canonical lengths of~$x$ and~$y$. In order to prove this, we remark that the canonical length of all the braids~$x^i$ and~$y^i$, for $i=1,\ldots, \beta(n)-1$, is bounded above by $\beta(n)\ell$.
By Theorem~\ref{bklbound} and Proposition~\ref{bla}, the number of iterations needed in order to find~$\widetilde{x}$ is linearly bounded with respect to this length $\beta(n)\ell$. Finally, each iteration of the operation~$\mathfrak s$ on a braid of canonical length~$\ell$ takes time~$O(\ell)$(see \cite{gebhardtgmalgo}).


Let $s=lcm(i_x,i_y)$. Since powers of rigid braids are again rigid, $x^s$ and~$y^s$ are conjugate to rigid braids. So all our algorithm has to do now is to calculate~$s$, and output
$\bar x=\tilde x^{\frac{s}{i_x}}$, $\bar y=\tilde y^{\frac{s}{i_y}}$, and $z_1=z_x$, $z_2=z_y$. This satisfies the requirements of~Theorem~\ref{algosecond}.
\end{proof}

\begin{proof}[Proof of Theorem~\ref{T:main}. ] We have to prove that the algorithm described in the introduction is valid and of complexity $O(\ell(x)^3)$. All the necessary ingredients are already at our disposal. Steps~(1) and~(3) are of complexity $O(\ell^2)$, as was shown in~\cite{calvezwiest}. Step~(2) is of complexity $O(\ell^3)$ (see Theorem 1 dans \cite{bggm3}). Step~(4) is of complexity $O(\ell^2)$, by Theorem~\ref{algosecond}. Finally, Theorem~4.11 of~\cite{gebhardtgmalgo} assures us that Algorithm~3 in~\cite{gebhardtgmalgo} correctly solves CDP and CSP for rigid braids of length at most~$\ell$ in time~$O(\ell\cdot \kappa)$, where~$\kappa$ denotes the cardinality of the SC of the input braids. Our Theorem~\ref{cardinalbound} now implies that step~(5) of our algorithm has complexity $O(\ell^3)$. Moreover, step~(5) gives the correct answer, because by~\cite{gmroots}, the relation $\bar x=\bar y^c$ for a braid~$c$ (i.e.\ $(x^s)^{z_1}=((y^s)^{z_2})^c$) is equivalent to the relation $x^{z_1}=y^{z_2 c}$, which is in turn equivalent to~$x$ being conjugate to~$y$ by~$z_2 c z_1^{-1}$.
\end{proof}


\end{document}